\documentclass[12pt]{amsart}

\usepackage{url}
\usepackage{amssymb} 
\usepackage{epsfig}
\usepackage{amsmath} 
\usepackage{graphicx}
\usepackage{psfrag}
\usepackage{color}
\newtheorem{thm}{Theorem}[section]
\newtheorem{dfn}[thm]{Definition}

\newtheorem{rmk}[thm]{Remark}
\newtheorem{cor}[thm]{Corollary}
\newtheorem{prop}[thm]{Proposition}
\newtheorem{lem}[thm]{Lemma}

\newcommand{\C}{{{\mathbb C}}}
\newcommand{\R}{{{\mathbb R}}}
\newcommand{\Q}{{{\mathbb Q}}}
\newcommand{\Z}{{{\mathbb Z}}}
\newcommand{\N}{{{\mathbb N}}}
\newcommand{\Hh}{{{\mathbb H}}}

\newcommand{\F}{{{\mathcal F}}}
\newcommand{\Ff}{{{\mathbb F}}}

\newcommand{\Ss}{{{\mathcal S}}}

\newcommand{\Cay}{{{\rm Cay}}}
\newcommand{\M}{{{\mathcal M}}}
\newcommand{\Aut}{{{\rm Aut}}}

\newcommand{\ord}{{{\rm ord}}}
\newcommand{\D}{{{\mathbb D}}}

\newcommand{\red}{\textcolor{red}}

\begin{document}

\title[Trivalent expanders and hyperbolic surfaces]{Trivalent expanders, 
$\mathbf{(\Delta-Y)}$-transformation, and hyperbolic surfaces}

\author{I. Ivrissimtzis, N. Peyerimhoff and A. Vdovina}

\address{Durham University, DH1 3LE, Great Britain}
\email{ioannis.ivrissimtzis@durham.ac.uk}
\email{norbert.peyerimhoff@durham.ac.uk} 
\address{University of Newcastle, NE1 7RU, Great Britain} 
\email{alina.vdovina@ncl.ac.uk}

\date{December 2015}

\begin{abstract}
  We construct a new family of trivalent expanders tessellating
  hyperbolic surfaces with large isometry groups. These graphs are
  obtained from a family of Cayley graphs of nilpotent groups via
  $(\Delta-Y)$-transformations. We compare this family with Platonic
  graphs and their associated hyperbolic surfaces and see that they
  are generally very different with only one hyperbolic surface in the
  intersection. Moreover, we study combinatorial, topological and
  spectral properties of our trivalent graphs and their associated
  hyperbolic surfaces.
\end{abstract}

\maketitle


\section{Introduction and statement of results}

In this article, we consider a family of surface tessellations with
interesting spectral gap properties. More precisely, we construct a
family of trivalent graphs $T_k$ ($k \ge 2$), which tessellate closed
hyperbolic surfaces $\Ss_k$ with large isometry groups, growing linear
on genus. The faces of these tessellations are regular hyperbolic
$2^{\lfloor \log_2 k \rfloor+2}$-gons and interior angles
$2\pi/3$. The graphs $T_k$ are ($\Delta-Y$)-transformations of Cayley
graphs of increasing $2$-groups $G_k$, and they form a family of
expanders (see Theorem \ref{thm:main} below for more details).

Using the Brooks-Burger transfer principle, the expansion properties
of the trivalent graphs $T_k \subset \Ss_k$ translate into a uniform
lower bound on the smallest positive Laplace eigenvalue
$\lambda_1(S_k)$ of the surfaces $\Ss_k$ (see Corollary
\ref{cor:lowerlambda1}). Prominent explicit examples where such an
interplay between groups, graphs and compact and non-compact
hyperbolic surfaces has been utilized, are finite quotients of
${\rm PSL}(2,\Z)$ and co-compact arithmetic lattices in
${\rm PSL}(2,\R)$ (see, e.g., Buser \cite{Bu2}, Brooks \cite{Br3}, and
Lubotzky \cite{Lub} and the references therein). While many finite
quotients of these examples are {\em simple}, all our finite groups
$G_k$ are {\em nilpotent} and very different in nature. We also show
that, while our graph $T_2$ in the surface $\Ss_2$ is dual to the
Platonic graph $\Pi_8$ (as defined below), there is no direct relation
between our family of graphs $T_k$ and these Platonic graphs $\Pi_N$
from $k \ge 3$ onwards. This fact is stated in Proposition
\ref{prop:isomt2s8thm:main}.

Another way to associate hyperbolic surfaces to the trivalent graphs
$T_k$ is to replace the vertices by hyperbolic regular $Y$-pieces and
to glue their boundary cycles in accordance with the edges between the
vertices of $T_k$, as explained in \cite{Bu1}. We denote the so
obtained hyperbolic surfaces by $\widehat \Ss_k$. For
these surfaces, Buser's results \cite{Bu1} yield a uniform lower bound
on $\lambda_1(\widehat \Ss_k)$ (see Corollary \ref{cor:lowerlambda2}).

Let us first give a brief overview over the construction of $T_k$ and
its spectral properties, and then go into further details. We start
with a sequence of $2$-groups $G_k$, following the construction in
\cite[Section 2]{PV}. Then we consider $6$-valent Cayley graphs $X_k$
of these groups and apply $(\Delta -Y)$-transformations in all
triangles of $X_k$, to finally obtain the trivalent graphs $T_k$. The
$(\Delta-Y)$-transformations are standard operations to simplify
electrical circuits, and were also used in \cite{BCdV} in connection
with Colin de Verdi{\'e}re's graph parameter. We have the following
relations between the spectra of the adjacency operators of the graphs
$X_k$ and $T_k$ which are proved in Section \ref{sec:mainres}. Note
that the spectra of the graphs $T_k$ are symmetric around the origin,
since these graphs are bipartite.

\begin{thm} \label{thm:main} Let $k \ge 2$. Then every eigenvalue
  $\lambda \in [-3,3]$ of $T_k$ gives rise to an eigenvalue $\mu =
  \lambda^2-3 \in [-3,6]$ on $X_k$. In particular, there exists a
  positive constant $C < 6$ such that
  \begin{itemize}
  \item[(i)] the graphs $X_k$ are $6$-valent expanders with
    spectrum in $[-3,C] \cup \{6\}$,
  \item[(ii)] the bipartite graphs $T_k$ are trivalent expanders with
    spectrum in $[-\sqrt{C+3},\sqrt{C+3}] \cup \{\pm 3 \}$.
  \end{itemize}
\end{thm}

The finite groups $G_k$ are constructed as follows. We start with the 
infinite group $\widetilde G$ of seven generators and seven relations:
\begin{equation} \label{eq:G}
\widetilde G = \langle x_0,\dots,x_6 \mid x_i x_{i+1} x_{i+3} \ 
\text{for $i=0,\dots,6$} \rangle, 
\end{equation}
where the indices are taken modulo $7$. As explained in \cite[Thm
3.4]{CMSZ}, this group acts simply transitively on the vertices of a
thick Euclidean building of type $\tilde A_2$. Let $S = \{x_0^{\pm
  1},x_1^{\pm 1},x_3^{\pm 1}\}$. We consider the index two subgroup $G
\le \widetilde G$, generated by $S$. (Note that $x_3 = x_1^{-1}
x_0^{-1}$.) We use a representation of the group $G$ by infinite
(finite band) upper triangular Toeplitz matrices with special
periodicity properties.  Let $G^k$ be the finite index normal subgroup
of elements in $G$ whose matrices have vanishing first $k$ upper
diagonals. The groups $G_k$ are then the finite quotients $G/G^k$. The
details of this construction are given in Section
\ref{subsec:faithrep}.

The {\em finite width conjecture} in \cite{PV} asks whether the groups
$G_k$ have another purely abstract group theoretical description via
the {\em lower exponent-$2$ series}
$$ G = P_0(G) \ge P_1(G) \ge P_2(G) \ge \cdots, $$
with 
\begin{equation} \label{eq:PkG}
P_k(G) = [P_{k-1}(G),G]P_{k-1}(G)^2
\end{equation} 
for $k \ge 1$: namely, do
we have $G^k = P_k(G)$ for $k \ge 1$ (see \cite[Conj. 1]{PV})? We
denote the order of $G_k$ by $2^{N_k}$. We know for $k \ge 1$ that
(see \cite[Cor. 2.3]{PV})
\begin{equation} \label{eq:N_k} 
N_k \ge 8 \lfloor k/3 \rfloor + 3 \cdot (k\, {\rm mod}\, 3) - 1, 
\end{equation}
and the finite width conjecture would imply that \eqref{eq:N_k} holds
with equality. MAGMA computations confirm this conjecture for all
indices up to $k = 100$. Henceforth, we use the same notation for the
elements $x_0,x_1,x_3$ in $G$ and their images in the quotients $G_k$,
for simplicity. Then $X_k = \Cay(G_k,S)$, and the trivalent graphs
$T_k$ are their $(\Delta-Y)$-transformations.

The graphs $T_k$ can be naturally embedded as tessellations into both
closed hyperbolic surfaces $\Ss_k$ and complete non-compact finite
area hyperbolic surfaces $\Ss_k^\infty$. The edges of the tessellation
are geodesic arcs and the vertices are their end points. The details
of these embeddings are presented in Section \ref{subsec:surfs}. The
following proposition describes the combinatorial properties of the
tessellations $T_k \subset \Ss_k$. The proof is given in Section
\ref{subsec:combprops}.

\begin{prop} \label{prop:combprops} Let $k \ge 2$. Then the generators
  $x_0, x_1, x_3$ of $G_k$ have all the same order $2^{n_k}$ with
  \begin{equation} \label{eq:n_k} 
  n_k = \lfloor \log_2 k \rfloor + 1. 
  \end{equation}
  Let $|G_k| = 2^{N_k}$ (with $N_k$ estimated from below in
  \eqref{eq:N_k}) and $V_k, E_k$ and $F_k$ denote the sets of
  vertices, edges, and faces of the tesselation $T_k \subset \Ss_k$,
  respectively. Then the isometry group of $\Ss_k$ has order $\ge
  2^{N_k}$, we have
  \begin{equation} \label{eq:VkEkFk} 
  |V_k| = 2^{N_k+1}, \quad |E_k| = 3 \cdot 2^{N_k} \quad \text{and} \
  \ |F_k| = 3 \cdot 2^{N_k-n_k}, 
  \end{equation}
  and all faces of $T_k \subset \Ss_k$ are regular hyperbolic
  $2^{n_k+1}$-gons with interior angles $2\pi/3$. Moreover, the genus
  of $\Ss_k$ is given by
  $$ g(\Ss_k) = 1 + 2^{N_k-n_k-1}(2^{n_k}-3). $$
\end{prop}

Since the hyperbolic surfaces $\Ss_k$ are closed, the spectrum of the
Laplace operator on $\Ss_k$ (with multiplicities) is a non-decreasing
sequence of real numbers $\lambda_k(\Ss_k)$, tending to infinity. In
this enumeration, we have $\lambda_0(\Ss_k) = 0$ and the corresponding
eigenfunction is constant. The following result is a consequence of
our Theorem \ref{thm:main} above and the Brooks-Burger transfer
principle. The details of the proof are given in Section
\ref{subsec:corlambda}. Note that the proof provides an explicit
estimate of the constant $C_1$ in Corollary \ref{cor:lowerlambda1} in
terms of the constant $C$ given in Theorem \ref{thm:main}.

\begin{cor} \label{cor:lowerlambda1} There is a positive constant $C_1
  > 0$ such that we have for the compact hyperbolic surfaces $S_k$ ($k
  \ge 2)$,
  $$ \lambda_1(S_k) \ge C_1. $$
\end{cor}

In Section \ref{subsec:confcomp}, we explain that the surfaces $\Ss_k$
are the {\em conformal compactifications} of the non-compact surfaces
$\Ss_k^\infty$. Moreover, the Cheeger constants of the graphs $T_k$
have a uniform lower positive bound since they are an expander
family. These facts imply via the Theorems 3.3 and 4.2 of \cite{BrM}
that the Cheeger constants of both families $\Ss_k$ and $\Ss_k^\infty$
have also uniform lower positive bounds.

\medskip

It is instructive to compare our tessellations $T_k \subset S_k$ to
the well studied tessellations of hyperbolic surfaces by {\em Platonic
  graphs} $\Pi_N$. It turns out that both families agree in one
tessellation (up to duality) but are, otherwise, very different. Let
us first define the graphs $\Pi_N$. Let $N$ be a positive integer
$\ge 2$. The vertices of $\Pi_N$ are equivalence classes
$[\lambda,\mu] = \{ \pm (\lambda,\mu) \}$ with
$$ \{ (\lambda,\mu) \in \Z_N \times \Z_N \mid \gcd(\lambda,\mu,N)=1 \}. $$
Two vertices $[\lambda,\mu]$ and $[\nu,\omega]$ are connected by an edge
if and only if
$$ \det \begin{pmatrix} \lambda & \nu \\ \mu & \omega \end{pmatrix} = 
\lambda \omega - \mu \nu = \pm 1. $$ 

Note that every vertex of $\Pi_N$ has degree $N$. These graphs can
also be viewed as triangulations of finite area surfaces $\Hh^2 /
\Gamma(N)$ by ideal hyperbolic triangles, where $\Hh^2 = \{ z = x + iy
\in \C \mid y > 0 \}$ denotes the upper half plane with hyperbolic
metric $\frac{dx^2+dy^2}{y^2}$ and 
\begin{equation} \label{eq:GammaN} 
  \Gamma(N) = \left\{ \gamma
    = \begin{pmatrix} a & b \\ c & d
  \end{pmatrix} \in {\rm PSL}(2,\Z) \mid \gamma \equiv \pm \begin{pmatrix}
    1 & 0 \\ 0 & 1 \end{pmatrix} \mod N \right\} 
\end{equation}
is the principal congruence subgroup of the modular group
$\Gamma = {\rm PSL}(2,\Z)$. The vertices of this triangulation are the
cusps of $\Ss^\infty(\Pi_N) = \Hh^2 / \Gamma(N)$.  


As mentioned above, $\Pi_8$ is isomorphic to the dual of $T_2$ in
$\Ss_2$. Since the valence of the dual graph $T_k^*$ is a power of
$2$, any isomorphism of $T_k^*$ with a Platonic graph $\Pi_N$ would
imply $N = 2^{n_k+1}$ with $n_k$ given in \eqref{eq:n_k}. However,
this leads to a contradiction {\em for all} $k \ge 3$. The next
proposition summarizes these results. The proof is given in Section
\ref{sec:isomT2Pi8}.

\begin{prop} \label{prop:isomt2s8thm:main} The Platonic graph $\Pi_8$
  is isomorphic to the dual of $T_2$ in the {\em unique} compact genus
  5 hyperbolic surface $\Ss_2$ with maximal automorphism group of
  order 192. For $k \ge 3$, there is no graph isomorphism between
  $T_k^*$ and $\Pi_N$, for any $N \ge 2$.
\end{prop}

It is easily checked via Euler's polyhedral formula that any
triangulation $X$ of a compact oriented surface $\Ss$ satisfies
$$ |E(X)| = 3(|V(X)|-2) + 6g(\Ss), $$
i.e., the number of edges $|E(X)|$ of every triangulation $X$ with at
least two vertices is $\ge 6g(\Ss)$.  Therefore, the ratio
$$ \frac{6g(\Ss)}{|E(X)|} \le 1 $$ 
measures the non-flatness of such a triangulation, i.e., how
efficiently the edges of $X$ are chosen to generate a surface of high
genus. Note that, for every $k \ge 2$, the dual graph $T_k^*$ can be
viewed as a triangulation of $\Ss_k$ and that the number of edges of
$T_k$ and $T_k^*$ coincide. Then we have the following asymptotic
result, proved in Section \ref{subsec:combprops}.

\begin{prop} \label{prop:asymptnonflat}
  We have
  \begin{equation} \label{eq:asympTk}
  \lim_{k \to \infty} \frac{6g(\Ss_k)}{|E(T_k^*)|} = 1, 
  \end{equation}
  where $E(T_k^*)$ denotes the set of edges of $T_k^*$.
\end{prop}

The other above-mentioned family $\widehat \Ss_k$ of surfaces associated to the graphs
$T_k$ can be viewed as tubes around $T_k$ with specific hyperbolic metrics.
The surfaces $\widehat \Ss_k$ form an infinite family of coverings. It
follows from \cite{Bu1} that their smallest positive Laplace eigenvalue
has also a uniform positive lower bound.

\begin{cor} \label{cor:lowerlambda2}
  The compact hyperbolic surfaces $\widehat \Ss_k$ ($k \ge 2$) have
  genus $1+|V_k|/2$ and isometry groups of order $\ge
  |V_k|/2$. They form a tower of coverings
  $$ \cdots \longrightarrow \widehat \Ss_{k+1} \longrightarrow \widehat 
  \Ss_k \longrightarrow \widehat \Ss_{k-1} \longrightarrow
  \cdots, $$ 
  where all the covering indices are powers of $2$. There is a positive
  constant $C_2 > 0$ such that we have, for all $k$,
  $$ \lambda_1(\widehat \Ss_k) \ge C_2. $$
\end{cor}

Corollary \ref{cor:lowerlambda2} is proved in Section
\ref{subsec:corlambda2}. There is a well-known classical result by
Randol \cite{Ran} which is, in some sense, complementary to this
corollary, namely, there exist finite coverings $\widetilde \Ss$ of
every closed hyperbolic surface $\Ss$ with arbitrarily small first
positive Laplace eigenvalue.

\medskip

\noindent {\bf Acknowledgement:} The authors are grateful to Peter
Buser, Shai Evra, and Hugo Parlier for several useful discussions and
information about the articles \cite{Breu}, \cite{BrM}, \cite{Bur},
and \cite{Bu1}. They also acknowledge the support of the EPSRC Grant
EP/K016687/1 ``Topology, Geometry and Laplacians of Simplicial
Complexes". NP and AV are grateful for the kind hospitality of the
Max-Planck-Institut f\"ur Mathematik in Bonn.

\section{Combinatorial properties of the tessellations 
$T_k \subset \Ss_k$}
\label{sec:combpropssurf}

Let $\widetilde G$ be the group defined in \eqref{eq:G}. It was shown
in \cite[Section 2]{PV} that the subgroup $G$, generated by $x_0,x_1$,
is an index two subgroup of $\widetilde G$. (Note that our group
$\widetilde G$ is denoted in \cite{PV} by $\Gamma$, which is reserved
for ${\rm PSL}(2,\Z)$ in this paper.) $G$ is explicitly given by $G =
\langle x_0,x_1 \mid r_1,r_2,r_3 \rangle$ with
\begin{eqnarray} \label{eq:r1r2r3}
r_1(x_0,x_1) &=& (x_1x_0)^3x_1^{-3}x_0^{-3}, \nonumber \\
r_2(x_0,x_1) &=& x_1 x_0^{-1} x_1^{-1} x_0^{-3} x_1^2 x_0^{-1} x_1 x_0 x_1, \\
r_3(x_0,x_1) &=& x_1^3 x_0^{-1} x_1 x_0 x_1 x_0^2 x_1^2 x_0 x_1 x_0. \nonumber
\end{eqnarray}

\subsection{A faithful matrix representation of $G$}
\label{subsec:faithrep}

Let us first recall the faithful representation of $G$ by infinite
upper triangular Toeplitz matrices, given in \cite{PV} and based on
representations introduced in \cite{CMSZ}. In fact, every $x \in G$
has a representation of the form
\begin{equation} \label{eq:matrix-x}
x = \begin{pmatrix} 
1 & a_{11} & a_{21} & \dots & a_{k1} & 0 & 0 & $\dots$ & $\dots$ \\ 
0 & 1 & a_{12} & a_{22} & \dots & a_{k2} & 0 & \ddots & \\
0 & 0 & 1 & a_{13} & a_{23} & \dots & a_{k3} & 0 & \ddots \\
\vdots & \ddots & 0 & 1 & a_{11} & a_{21} & \dots & a_{k1} & \ddots \\
\vdots & & \ddots & \ddots & 1 & a_{12} & a_{22} & \dots & \ddots \\
\vdots & & & \ddots & \ddots & 1 & a_{13} & a_{23} & \ddots \\
\vdots & & & & \ddots & \ddots & \ddots & \ddots & \ddots
\end{pmatrix}, 
\end{equation}
where each element $a_{ij}$ is in the set $M(3,\Ff_2)$ of $(3 \times
3)$-matries with entries in $\Ff_2$, and $0$ and $1$ stand for the
zero and identity matrix in $M(3,\Ff_2)$. Note the periodic pattern in
the upper diagonals of the matrix, i.e., the $j$-th upper diagonal is
uniquely determined by the first three entries $a_j =
(a_{j1},a_{j2},a_{j3})$, which can be understood as a $(3 \times
9)$-matrix with values in $\Ff_2$. We use the short-hand notation
$M_0(a_1,a_2,\dots,a_k)$ for the matrix in \eqref{eq:matrix-x}. If the
first $l$ upper diagonals in \eqref{eq:matrix-x} vanish, we also write
$M_l(a_{l+1},\dots,a_k)$. Let $G^k$ be the subgroup of all elements $x
\in G$ with vanishing first $k$ upper diagonals in their matrix
representation. It follows from the structure of these matrices that
$G^k$ is normal and that the quotient group $G_k = G/G^k$ is a
$2$-group, i.e., nilpotent.

Recall that we use the same notation for the generators $x_0,x_1,x_3$
of $G$ and their images in the quotient $G_k$. We will see later that
the faces of the tessellation $T_k \subset \Ss_k$ are determined by
the orders of these generators in $G_k$.  We will now determine these
orders. Let {\small
\begin{align*}
& \alpha_0 = \begin{pmatrix} 0 & 0 & 0 & 0 & 0 & 0 & 0 & 0 & 0\\ 
                   0 & 0 & 1 & 0 & 0 & 1 & 0 & 0 & 1\\ 
                   0 & 1 & 1 & 0 & 1 & 1 & 0 & 1 & 1 \end{pmatrix},
& \beta_0 = \begin{pmatrix} 0 & 0 & 0 & 0 & 0 & 0 & 0 & 0 & 0\\ 
                   0 & 1 & 1 & 0 & 1 & 1 & 0 & 1 & 1\\ 
                   0 & 1 & 0 & 0 & 1 & 0 & 0 & 1 & 0 \end{pmatrix},\\[.2cm]
& \alpha_1 = \begin{pmatrix} 0 & 0 & 0 & 0 & 1 & 1 & 0 & 1 & 0\\ 
                   0 & 1 & 0 & 1 & 0 & 0 & 0 & 0 & 1\\ 
                   1 & 1 & 1 & 0 & 0 & 0 & 0 & 1 & 0 \end{pmatrix},
& \beta_1 = \begin{pmatrix} 0 & 0 & 0 & 0 & 1 & 1 & 0 & 1 & 0\\ 
                   0 & 1 & 0 & 1 & 0 & 0 & 0 & 0 & 1\\ 
                   1 & 1 & 1 & 0 & 0 & 0 & 0 & 1 & 0 \end{pmatrix},\\[.2cm]
& \alpha_3 = \begin{pmatrix} 0 & 0 & 0 & 0 & 1 & 1 & 0 & 1 & 0\\ 
                   0 & 1 & 1 & 1 & 0 & 1 & 0 & 0 & 0\\ 
                   1 & 0 & 0 & 0 & 1 & 1 & 0 & 0 & 1 \end{pmatrix},
& \beta_3 = \begin{pmatrix} 0 & 0 & 0 & 0 & 0 & 1 & 0 & 1 & 1\\ 
                   1 & 1 & 0 & 0 & 1 & 1 & 0 & 0 & 0\\ 
                   0 & 1 & 1 & 0 & 0 & 1 & 1 & 0 & 0 \end{pmatrix}.
\end{align*}
} 
Then we have $x_i = M_0(\alpha_i,\dots)$ for $i=0,1,3$, and we obtain
the following fact about the leading diagonal of their $2$-powers.

\begin{lem} \label{lem:xipow} We have for $i \in \{0,1,3\}$ and $l \ge
  0$:
  \begin{equation} \label{eq:xi^2l}
  x_i^{2^l} =
  \begin{cases} M_{2^l-1}(\alpha_i,\dots), & \text{if $l$ is even},\\
                M_{2^l-1}(\beta_i,\dots), & \text{if $l$ is odd}.
  \end{cases}
  \end{equation}
  This implies, in particular, for $k \in \N$ that the order of $x_i$
  in $G_k$ is $2^{n_k}$ with $n_k$ given in \eqref{eq:n_k}.
\end{lem}

\begin{proof} Since $G_k$ is a $2$-group, $\ord_{G_k}(x_i)$ has to be
  a power of $2$. The formulas \eqref{eq:xi^2l} follow via a
  straightforward calculation using Prop. 2.5 in \cite{PV}. This
  implies that $\ord_{G_k}(x_i) = 2^l$ if and only if $2^{l-1} \le k <
  2^l$, i.e., $l = \lfloor \log_2 k \rfloor +1 = n_k$.
\end{proof}

\subsection{The surfaces $\Ss_k$ and $\Ss_k^\infty$ via covering
  theory}
\label{subsec:surfs}

Let $X_k$ be the Cayley graph $\Cay(G_k,S)$. 
We will now explain how to construct the closed hyperbolic surfaces
$\Ss_k$: We start with an orbifold $\Ss_0$ by gluing together two {\em
  compact} hyperbolic triangles ${\mathcal T}_1, {\mathcal T}_2
\subset \Hh^2$ with angles $\pi/\ord_{G_k}(x_0), \pi/\ord_{G_k}(x_1)$
and $\pi/\ord_{G_k}(x_3)$ along their corresponding sides. Both
triangles are equilateral since, by Lemma \ref{lem:xipow}, we have
$\ord_{G_k}(x_0)= \ord_{G_k}(x_1)=\ord_{G_k}(x_3) = 2^{n_k}$. It is
useful to think of the two triangles ${\mathcal T}_1$ and ${\mathcal
  T}_2$ in $\Ss_0$ to be coloured black and white, respectively. Let
$P_0, P_1, P_2 \in \Ss_0$ be the singular points (i.e., the identified
vertices of the triangles ${\mathcal T}_1$ and ${\mathcal T}_2$ in
$\Ss_0$) and $Q \in \Ss_0$ be the center of the white triangle
${\mathcal T}_1$. Note that $\Ss_0 \backslash \{ P_0,P_1,P_2 \}$
carries a hyperbolic metric induced by the triangles ${\mathcal T}_1,
{\mathcal T}_2$.  Choose a geometric basis
$\gamma_0,\gamma_1,\gamma_2$ of the fundamental group
$\pi_1(\Ss_0\backslash\{P_0,P_1,P_2\},Q)$ such that $\gamma_i$ is a
simple loop (starting and ending at $Q$) around the singular point
$P_i \in \Ss_0$ and $\gamma_0 \gamma_1 \gamma_2 = e$. Note that
$\pi_1(\Ss_0\backslash\{P_0,P_1,P_2\},Q)$ is a free group in the
generators $\gamma_0,\gamma_1$. The surjective homomorphism
$$ \Psi: \pi_1(\Ss_0\backslash\{P_0,P_1,P_2\},Q) \to G_k, $$
given by $\Psi(\gamma_0) = x_0$, $\Psi(\gamma_1) = x_1$ and
$\Psi(\gamma_2) = x_3$, induces a branched covering $\pi: \Ss_k \to
\Ss_0$, by Riemann's existence theorem (see \cite[Thms. 4.27 and
4.32]{Vo} or \cite[(17)]{BCG}) with all ramification indices equals
$2^{n_k}$ and, therefore, the closed surface $\Ss_k$ carries a
hyperbolic metric such that the restriction
$$ \pi: \Ss_k \backslash \pi^{-1}(\{P_0,P_1,P_2\}) \to \Ss_0 \backslash
\{P_0,P_1,P_2\} $$ is a Riemannian covering. The surface $\Ss_k$ is a
Belyi surface since it is a branched covering over $\Ss_0$ ramified at
the three points $P_0, P_1, P_2$. Moreover, $\Ss_k$ is tessellated by
$2 |G_k| = 2^{N_k+1}$ equilateral hyperbolic triangles, half of them
black and the others white. Hurwitz's formula yields
\begin{equation} \label{eq:gSinf}
g(\Ss_k) = 1 + \frac{1-\mu_k}{2} |G_k|, 
\end{equation}
where
\begin{equation} \label{eq:mu}
  \mu_k = \frac{1}{\ord(x_0)} + \frac{1}{\ord(x_1)} + \frac{1}{\ord(x_3)}
  = \frac{3}{2^{n_k}}.
\end{equation}
Recall that the orders $2^{n_k}$ of $x_i$ ($i = 0,1,3$) were
given in Lemma \ref{lem:xipow}. In the case $k=2$ we have $|G_2|
= 32$ and $\ord(x_0)=\ord(x_1)=\ord(x_3)=4$, which leads to
$$ g(\Ss_2) = 1 + \frac{1}{8} \cdot 32 = 5. $$

$G_k$ acts simply transitive on the black triangles of $\Ss_k$. Let $V
= \pi^{-1}(Q)$ and $V_{\rm white}, V_{\rm black} \subset V$ be the
sets of centers of white and black triangles, respectively. Choose a
reference point $Q_0 \in V_{\rm white}$, and identify the vertices of
the Cayley graph $X_k=\Cay(G_k,S)$ with the points in $V_{\rm white}$
by $G_k \ni h \mapsto hQ_0 \in V_{\rm white}$. Then two adjacent
vertices in $X_k$ are the centers of two white triangles which share a
black triangle as their common neighbour. The corresponding edge in
$X_k$ can then be represented by the minimal geodesic passing through
these three triangles and connecting these two vertices. Moreover,
$G_k$ acts on the surface $\Ss_k$ by isometries and we have $\Ss_0 =
\Ss_k/G_k$, i.e., the isometry group of $\Ss_k$ has order $\ge |G_k| =
2^{N_k}$.

We could also start the process by gluing together two {\em ideal}
hyperbolic triangles ${\mathcal T}_1^\infty$ and ${\mathcal
  T}_2^\infty$, coloured black and white, along their corresponding
edges. Each edge of ${\mathcal T}_i^\infty$ ($i=1,2$) has a unique
intersection point {\em (tick mark)} with the incircle of the
triangle, and these tick-marks of corresponding edges of ${\mathcal
  T}_1^\infty$ and ${\mathcal T}_2^\infty$ are identified under the
gluing. The resulting surface $\Ss_0^\infty$ is topologically a
$3$-punctured sphere, carrying a complete hyperbolic metric of finite
volume, and the same arguments as above then lead to an embedding of
the graphs $X_k$ into complete non-compact finite area hyperbolic
surfaces $\Ss_k^\infty$, triangulated by ideal black and white
triangles, where the vertices of $X_k$ correspond to the centers of
the white ideal triangles in $\Ss_k^\infty$, and we have $\Ss_0^\infty
= \Ss_k^\infty / G_k$.

\begin{figure}[h]
  \psfrag{x0}{\tiny $\red{x_0}$}
  \psfrag{x1}{\tiny $\red{x_1}$}
  \psfrag{x3}{\tiny $\red{x_3}$}
  \begin{center}      
     \includegraphics[width=0.495\textwidth]{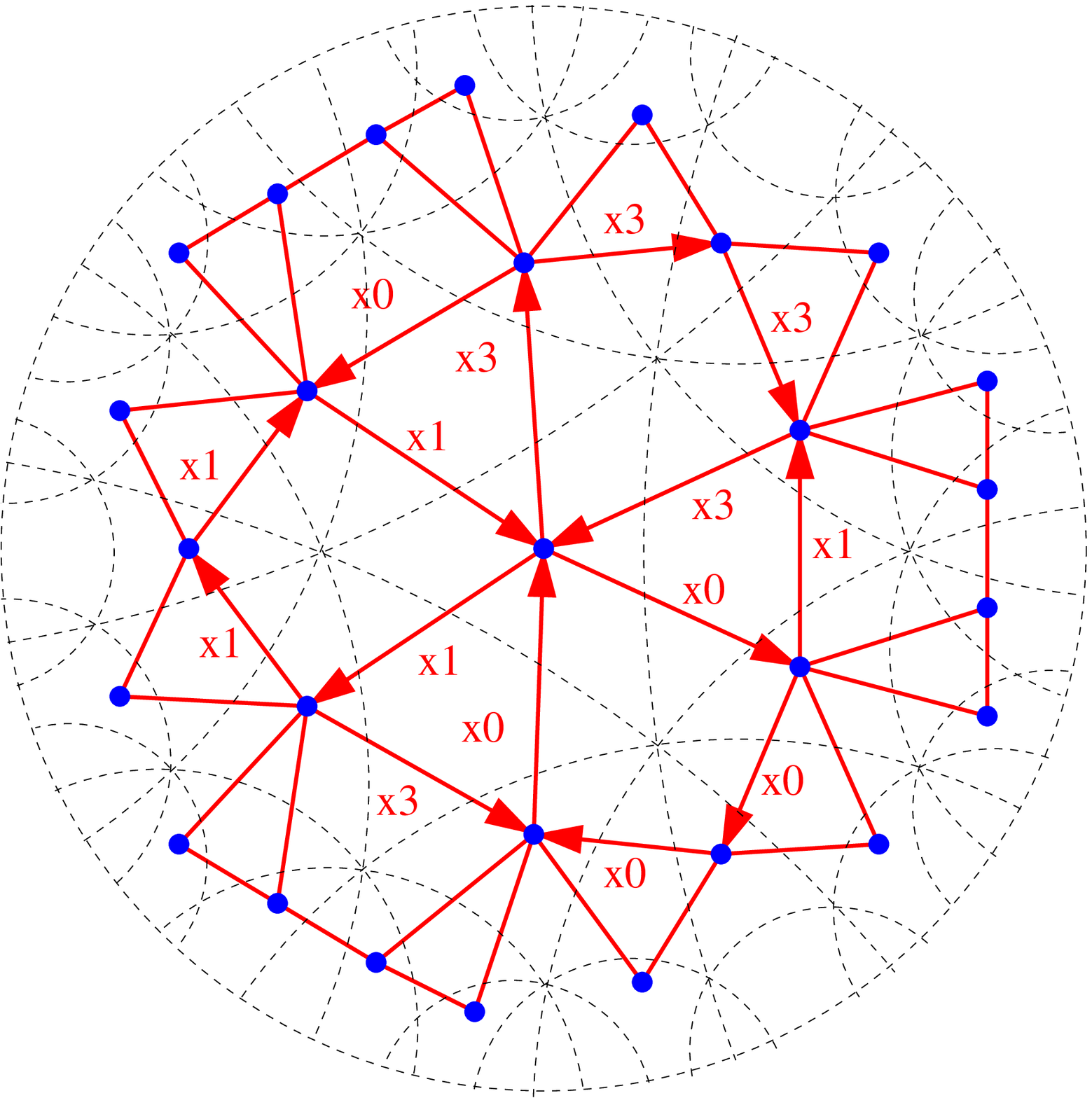}
     \includegraphics[width=0.495\textwidth]{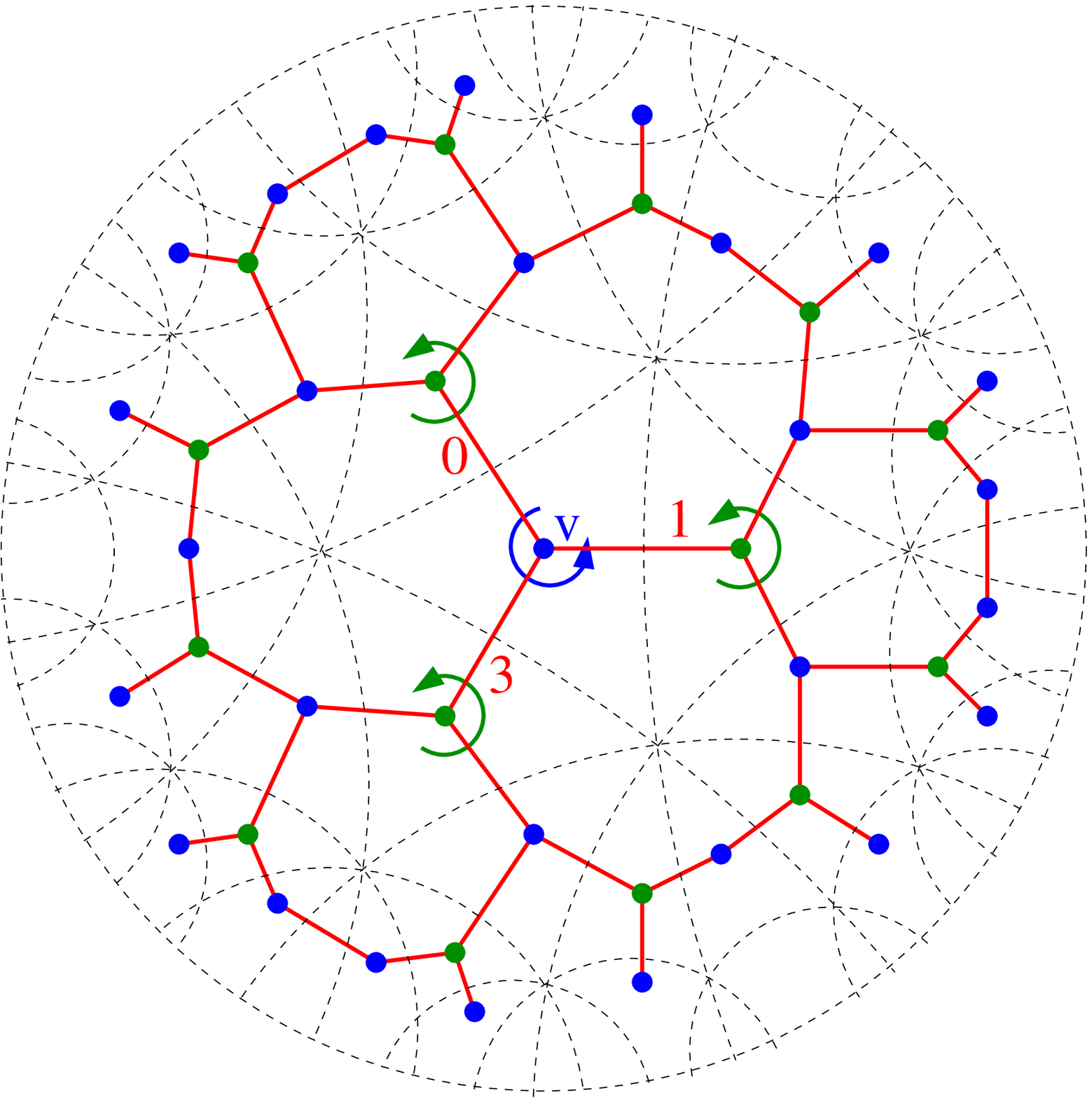}
  \end{center}
  \caption{The lifts of the Cayley graph $X_2$ (left) and of the
    $(\Delta-Y)$-transformation $T_2$ (right) to the Poincar{\'e} unit
    disc $\D^2$}
   \label{fig:DYtrafo}
\end{figure}

\subsection{From the Cayley graphs $X_k$ to the trivalent graphs
  $T_k$}
\label{subsec:XkTK}

For the transition from $X_k$ to $T_k$ we use the
$(\Delta-Y)$-transformation.  In this transformation, we add a new
vertex $v$ for every combinatorial triangle of the original graph,
remove the three edges of this triangle and replace them by three
edges connecting $v$ with the vertices of this triangle. We apply this
rule to our graph $X_k$ and obtain a graph $T_k$, which we can view
again as an embedding in $\Ss_k$ with the following properties: The
vertex set of $T_k$ coincides with $V$, and there is an edge (minimal
geodesic segment) connecting every black/white vertex in $V$ with the
vertices in the three neighbouring white/black triangles. The best way
to illustrate this transformation is to present it in the universal
covering of the surface $\Ss_k$, i.e., the Poincar{\'e} unit disc
$\D^2$ (see Figure \ref{fig:DYtrafo}, the new vertices replacing every
triangle are green). Note that $T_k$ has twice as many vertices as
$X_k$, which shows that the isometry group of the above compact
surface $\Ss_k$ has order $\ge |G_k| = |V_k|/2$, where $V_k$ denotes
the vertex set of $T_K$. Moreover, $T_k$ is the dual of the
triangulation of $\Ss_k$ by the above-mentioned compact black and
white triangles.

\subsection{A direct construction of $\Ss_k$ and $\Ss_k^\infty$ from
  $T_k$} \label{subsec:dirconst}
There is another method to obtain the hyperbolic surfaces
$\Ss_k$ and $\Ss_k^\infty$ using the construction in \cite[Section
4]{BrM} (see also \cite[Chapter 1]{Man}). The start data are our
trivalent graphs $T_k$ with a suitable orientation.

\begin{dfn} An {\em orientation ${\mathcal O}$} on a trivalent graph
  $T$ is a choice, at each vertex $v$ of $T$, of a cyclic ordering of
  the three edges emanating from it.
\end{dfn}

Let us first introduce an orientation ${\mathcal O}_k$ on $T_k$. We
start with the Cayley graph $X_k$ and orient its edges such that they
only carry the Cayley graph labels $x_0, x_1$ and $x_3$, and not their
inverses (see the Figure \ref{fig:DYtrafo} on the left). Every
triangle in $X_k$ forms then an oriented cycle with consecutive labels
$x_0, x_1, x_3$. This orientation induces an orientation on the new
green vertex in $T_k$ corresponding to this triangle, as illustrated
by the oriented green circular arcs in the Figure \ref{fig:DYtrafo} on
the right. A blue vertex $v$ of $T_k$ stems from a vertex of $X_k$,
and we can give the labels $0,1,3$ to the three edges in $T_k$
emanating from $v$, agreeing with the label of the edge in the
corresponding triangle of $X_k$ not adjacent to $v$ (see again Figure
\ref{fig:DYtrafo} for illustration). The orientation of the three
edges emanating from $v$ in $T_k$ (illustrated by an oriented blue
circular arc) is then given by the cyclic ordering $0,3,1$.

Now we follow the explanations in \cite[Sections 1.1-1.4]{Man}
closely. Let ${\mathcal T} \subset \D^2$ be an oriented compact
equilateral hyperbolic triangle with interior angles $\pi/2^{n_k}$.  We
refer to the mid-points of the sides of ${\mathcal T}$ as {\em
  tick-marks}.  The orientation of ${\mathcal T} \subset \D^2$ induces
a cyclic order on these tick-marks. Connect the center of ${\mathcal
  T}$ with the three tick-marks by geodesic arcs and assume that these
arcs are coloured red. Then we paste a copy of ${\mathcal T} \subset
\D^2$ on each vertex $v$ of $T_k$ such that its center agrees with
$v$, its tick-marks agree with the mid-points of the edges of $T_k$
emanating from $v$, and that the cyclic orders of these egdes and of
the tick-marks agree. Observe that, even though the mid-points of
adjacent sides of triangles meet up at mid-points of edges of $T_k$,
we have not yet identified the sides of these triangles. This
identification is made in such a way that the orientations of adjacent
triangles match up. The resulting hyperbolic surface $\Ss_k$ carries
then a global orientation and the union of the red geodesic arcs from
their mid-points to their tick-marks in the triangles provide an
embedding of the graph $T_k$ into this surface such that the faces are
regular $2n_k$-gons. 

The complete non-compact finite area hyperbolic surfaces
$\Ss_k^\infty$ are obtained in the same way by starting instead with
an oriented ideal hyperbolic triangle ${\mathcal T} \subset \D^2$ with
tick-marks. As explained at the end of \cite[Section 1.4]{Man}, the
cusps of $\Ss_k^\infty$ are then in bijection with the left-hand-turn
pathes in $(T_k,{\mathcal O}_k)$. This alternative construction of
$\Ss_k$ and $\Ss_k^\infty$ is later useful in Section
\ref{subsec:confcomp}.


\subsection{Proofs of Propositions \ref{prop:combprops} and \ref{prop:asymptnonflat}}
\label{subsec:combprops}

\begin{proof}[Proof of Proposition \ref{prop:combprops}]
  The orders of $x_0, x_1, x_3 \in G_k$ were given in Lemma
  \ref{lem:xipow}. It was explained in Section \ref{subsec:surfs} that
  the isometry group of $\Ss_k$ has order $\ge |G_k| = 2^{N_k}$, and
  in Section \ref{subsec:XkTK} that $|V_k| = 2|G_k| =
  2^{N_k+1}$. Lemma \ref{lem:xipow} implies that the faces of the
  triangulation $T_k \subset \Ss_k$ are regular $2^{n_k+1}$-gons,
  which yields $2|E_k| = 3|V_k| = 2^{n_k+1}|F_k|$, proving
  \eqref{eq:VkEkFk}. The genus $g(\Ss_k)$ can be derived either from
  Hurwitz's formula \eqref{eq:gSinf} or from the Euler characteristic
  $\chi(\Ss_k) = |V_k| - |E_k| + |F_k|$.  This finishes the proof of
  Proposition \ref{prop:combprops}.
\end{proof}

\begin{rmk} The number of cusps of the surface $\Ss_k^\infty$ agrees
  with the number of faces of the tessellation $T_k \subset
  \Ss_k$. For example, the genus five surface $\Ss_2$ mentioned in
  Section \ref{subsec:surfs} is tessellated into $24$ octagons, and
  the surface $\Ss_2^\infty$ has, therefore, $24$ cusps.
\end{rmk}

\begin{proof}[Proof of Proposition \ref{prop:asymptnonflat}]
  We conclude from \eqref{eq:gSinf}, $|V_k| =
  2|G_k|$, $|E(T_K^*)| = |E_k|$, and the trivalence of $T_k$ that
  $$ \frac{6g(\Ss(T_k))}{|E(T_k^*)|} = 6 \,
  \frac{1+(1-\mu_k)|V_k|/4}{3 |V_k|/2}. $$ Note that $|V_k| = 2|G_k| =
  2^{N_k+1} \to \infty$ because of \eqref{eq:N_k}, which implies that
  $$ \lim_{k \to \infty} \frac{6g(\Ss_k)}{|E(T_k^*)|} = 
  1 - \lim_{k\to\infty} \mu_k. $$ 
  Recall from \eqref{eq:mu} and \eqref{eq:n_k} that $\mu_k =
  3/2^{n_k} \to 0$ as $k \to \infty$, finishing the proof of
  \eqref{eq:asympTk}.
\end{proof}

\section{Spectral properties of the graphs $X_k$ and $T_k$}
\label{sec:mainres}


In this section, we establish the expander properties of $X_k$ and
$T_k$ and relations between their eigenfunctions and eigenvalues,
which proves Theorem \ref{thm:main}. We also investigate Ramanujan
properties of these families of graphs. 

\subsection{Precise relations between eigenfunctions and 
eigenvalues}

As before, let $\widetilde G$ be the group defined in \eqref{eq:G} and
$G$ be the index two subgroup generated by $S$. Then both groups
$\widetilde G$ and $G$ have Kazhdan property (T) (see \cite[Section
3]{PV}). Using \cite[Prop. 3.3.1]{Lub}, we conclude that the Cayley
graphs $X_k = {\rm Cay}(G_k,S)$ are expanders.

The adjacency operator $A_X$, acting on functions on the vertices of a
graph $X$, is defined as
$$ A_Xf(v) = \sum_{w \sim v} f(w), $$
where $w \sim v$ means that the vertices $v$ and $w$ are adjacent.  It
is easy to see that the eigenvalues of the adjacency operator of a
finite $n$-regular graph lie in the interval $[-n,n]$. 

Recall that the set $V(X_k)$ of vertices of our $6$-valent graph $X_k$
is a subset of the vertex set $V_k$ of the trivalent graph $T_k$. We
have the following relations between the eigenfunctions of the
adjacency operators on $X_k$ and $T_k$.

\begin{thm} \label{thm:releigfunc}
  \begin{itemize}
  \item[(a)] Every eigenfunction $F$ on $T_k$ to an eigenvalue
    $\lambda \in [-3,3]$ gives rise to an eigenfunction $f$ to the
    eigenvalue $\mu = \lambda^2 - 3 \in [-3,6]$ on $X_k$ (with
    $f(v) = F(v)$ for all $v \in V(X_k)$).
  \item[(b)] Every eigenfunction $f$ on $X_k$ to an eigenvalue
    $\mu \in [-6,6] - \{-3\}$ gives rise to two eigenfunctions $F_\pm$
    to the eigenvalues $\pm \sqrt{\mu + 3}$ on $T_k$ with
    $$ F_\pm(v) = \begin{cases} f(v) & \text{if $v \in V(X_k)$,} \\
      \pm \frac{1}{\sqrt{\mu+3}} \sum_{w \sim v} f(w) & \text{if $v \in
          V_k-V(X_k)$.} \end{cases}
    $$
  \item[(c)] An eigenfunction $f$ on $X_k$ to the eigenvalue $-3$
    gives rise to an eigenfunction $F$ to the eigenvalue $0$ of $T_k$ with
    $$ F(v) = \begin{cases} f(v) & \text{if $v \in V(X_k)$,} \\
      0 & \text{if $v \in V_k-V(X_k)$,} \end{cases}
    $$
    if and only if we have, for all triangles $\Delta \subset V(X_k)$,
    $\sum_{v \in \Delta} f(v) = 0$.
  \end{itemize}
\end{thm}

In the following proof, we use $\sim$ for adjacency in $T_k$ and
$\sim_{X_k}$ for adjacency in $X_k$. Moreover, $d_{T_k}$ denotes the
combinatorial distance function on the vertex set $V_k$ of $T_k$.

\begin{proof}
  (a) Let $f$ and $F$ be two functions on $X_k$ and $T_k$, related
  by $f(v) = F(v)$ for all $v \in V(X_k)$. Then
  $$ A_{X_k}f(v) = \sum_{w \sim_{X_k} v} f(w) = 
  \sum_{d_{T_k}(w,v) = 2} F(w) = (A_{T_k})^2 F(v) - 3 F(v), $$ which
  can also be written as $A_{X_k} = (A_{T_k})^2 - 3$. This implies
  immediately the connection between the eigenfunctions and
  eigenvalues.

  (b) Let $A_{X_k}f = \mu f$ and $F_\pm$ be defined as in the theorem.
  Let $\lambda = \pm \sqrt{\mu+3}$. Then we have for $v \in V(X_k)$:
  \begin{eqnarray*}
    A_{T_k}F_\pm(v) &=& \sum_{w \sim v} F_\pm(w) = \frac{1}{\lambda} \sum_{w \sim v}
    \sum_{x \sim w} F_\pm(x)\\
    &=& \frac{1}{\lambda} \left( 
      \sum_{w \sim_{X_k} v} f(w) + 3 f(v) \right)
    = \frac{\mu+3}{\lambda} f(v) = \lambda F_\pm(v), 
  \end{eqnarray*}
  and for $v \in V_k - V(X_k)$:
  $$ A_{T_k}F_\pm(v) = \sum_{w \sim v} F_\pm(w) = 
  \lambda \left( \frac{1}{\lambda} \sum_{w \sim v} f(w) \right) =
  \lambda F_\pm(v). $$ 

  Note that $1/\lambda$ is well defined since $\mu \neq -3$ and,
  therefore, $\lambda = \pm \sqrt{\mu+3} \neq 0$.

  (c) In the case of $\mu = -3$ we have $\lambda=0$, and 
  $$ A_{T_k} F(v) = \sum_{w \sim v} F(w) = 0 $$
  holds trivially for $v \in V(X_k)$. For all vertices $v \in V_k -
  V(X_k)$, the conditions
  $$ 0 = A_{T_k} F(v) = \sum_{w \sim v} f(v) $$
  translate into the condition that the summation of $f$ over the
  vertices of every triangle in $X_k$ must vanish.
\end{proof}

An immediate consequence of Theorem \ref{thm:releigfunc} is that the
expander property of the family $X_k$ carries over to the graphs $T_k$
(with the spectral bounds given in Theorem \ref{thm:main}). Moreover,
the spectrum of $X_k$ cannot contain eigenvalues in the interval
$[-6,-3)$, since this would lead to non-real eigenvalues of
$T_k$. Therefore, these arguments also complete the proof of Theorem
\ref{thm:main}.

\begin{rmk} It would be interesting to find an explicit value for the
  constant $C > 0$ in Theorem \ref{thm:main}. This would be possible
  if we were able to estimate the Kazdhan constant of the index two
  subgroup $G$ of $\widetilde G$ (with respect to some choice of
  generators). While the Kazhdan constant of $\widetilde G$ with
  respect to the the standard set of seven generators was explicitly
  computed in \cite{CMS}, it seems to be a difficult and challenging
  question to obtain an explicit estimate for the Kazhdan constant of
  the subgroup $G$.
\end{rmk}

\subsection{Ramanujan properties}

Recall that a finite $n$-regular graph $X$ is Ramanujan if all
non-trivial eigenvalues $\lambda \neq \pm n$ lie in the interval
$[-2\sqrt{n-1},2\sqrt{n-1}]$. Since the $6$-regular graphs $X_k$ are
Cayley graphs of quotients of the group $G$ with property (T), not all
of these graphs can be Ramanujan (see \cite[Prop. 4.5.7]{Lub}). MAGMA
computations provide the following numerical results:

\medskip

\begin{table}[htbp]
\begin{tabular}{l|l|l}
graph & number of vertices & largest non-trivial eigenvalue \\
\hline\rule[1.8mm]{0cm}{2mm}
$X_2$ & 32 &   2.828427\dots \\ \rule[1.5mm]{0cm}{2mm}
$X_3$ & 128 &  4.340172\dots \\ \rule[1.5mm]{0cm}{2mm}
$X_4$ & 1024 & 4.475244\dots \\ \rule[1.5mm]{0cm}{2mm}
$X_5$ & 8192 & 5.160252\dots
\end{tabular}
\end{table}

This implies that only $X_2$ and $X_3$ are Ramanujan; their largest
non-trivial eigenvalue needs to be $< 2 \sqrt{5} = 4.472135\dots$,
which is no longer true for $k = 4$. Moreover, since $X_{k+1}$ is a
lift of $X_k$, the spectrum of $X_k$ is contained in the spectrum of
$X_{k+1}$.

Next, we consider Ramanujan properties of the graphs $T_k$.  Theorem
\ref{thm:releigfunc} implies that $T_k$ is Ramanujan if and only if
the largest non-trivial eigenvalue of $X_k$ is $\le 5$. Therefore, the
above numerical results imply that only $T_2, T_3, T_4$ are Ramanujan.
Here are the numerical results for the largest {\em non-trivial} eigenvalues
$\lambda_1(T_k)$ of the first $T_k$'s: \medskip

\begin{table}[htbp]
\begin{tabular}{l|l|l}
graph & number of vertices & $\lambda_1(T_k)$ \\
\hline\rule[1.8mm]{0cm}{2mm}
$T_2$ & 64 &   2.414213\dots \\ \rule[1.5mm]{0cm}{2mm}
$T_3$ & 256 &  2.709275\dots \\ \rule[1.5mm]{0cm}{2mm}
$T_4$ & 2048 & 2.734089\dots \\ \rule[1.5mm]{0cm}{2mm}
$T_5$ & 16384 & 2.856615\dots
\end{tabular}
\end{table}

Note that the spectrum of $T_k$ is symmetric around the origin since
the graphs $T_k$ are bipartite. The {\em spectral gap of $T_k$} is
defined as the value $\sigma(T_k) = 3 - \lambda_1(T_k)$ and can
be described variationally as
\begin{equation} \label{eq:sigvarTk}
\sigma(T_k) = \inf \left\{ \frac{\sum_{\{v,w\} \in E_k} 
(F(v)-F(w))^2}{\sum_{v \in V_k} (F(v))^2} \Bigg\vert \sum_{v \in V_k} F(v) = 0 
\right\}. 
\end{equation}

\section{Comparison with Platonic graphs}
\label{sec:isomT2Pi8}

\subsection{Basics about Platonic graphs}

We first recall a few important facts about the Platonic graphs
$\Pi_N$ and the surfaces $\Ss^\infty(\Pi_N) = \Hh^2/\Gamma(N)$. For
more details see, e.g., \cite{ISi}. Let $\F$ be the Farey tessellation
of the hyperbolic upper half plane $\Hh^2$, and let $\Omega(\F)$ be
the set of oriented geodesics in $\F$. Recall that the {\em Farey
  tessellation} is a triangulation of $\Hh^2$ with vertices on the
line at infinity $\R \cup \{ \infty \}$, namely, the subset of
extended rationals $\Q \cup \{ \infty \}$. Two rational vertices with
reduced forms $a/c$ and $b/d$ are joined by an edge, a geodesic of
$\Hh^2$, if and only if $ad-bc = \pm 1$ (see \cite[Fig. 1]{ISi} for an
illustration of the Farey tessellation). The group of conformal
transformations of $\Hh^2$ that leave $\F$ invariant is the modular
group $\Gamma = {\rm PSL}(2,\Z)$, which acts transitively on
$\Omega(\F)$. The {\em principal congruence subgroups} $\Gamma(N)$ are
the normal subgroups given in \eqref{eq:GammaN}.

It is well known (see, e.g, \cite{ISi}) that $\F/\Gamma(N)$ and
$\Pi_N$ are isomorphic, and $\F/\Gamma(N)$ is a triangulation of the
surface $\Ss^\infty(\Pi_N) = \Hh^2/\Gamma(N)$ by ideal triangles (the
vertices are, in fact, the cusps of $\Ss^\infty(\Pi_N)$). The
tessellation $\Pi_N \subset \Ss^\infty(\Pi_N)$ can be interpreted as a
{\em map} $\M_N$ in the sense of Jones/Singerman \cite{JSi}. The group
$\Aut(\M_N)$ of automorphisms of $\M_N$ is the group of orientation
preserving isometries of $\Ss^\infty(\Pi_N)$ preserving the
triangulation. As $\Gamma(N)$ is normal in $\Gamma$, we have that the
map $\M_N$ is {\em regular}, meaning that $\Aut(\M_N)$ acts
transitively on the set of directed edges of $\Pi_N$ (see \cite[Thm
6.3]{JSi}). Moreover, by \cite[Thm 3.8]{JSi},
$$ \Aut(\M_N) \cong \Gamma/\Gamma(N) \cong {\rm PSL}(2,\Z_N). $$
(Note that in the case of a prime power $N = p^r$, ${\rm PSL}(2,\Z_N)$
is the group defined over the ring $\Z_N=\Z/(N \Z)$ and not over the
field ${\mathbb F}_q$ with $q=p^r$ elements.) Let $N \ge 7$. Noticing
that all vertices of $\Pi_N$ have degree $N$, we obtain a smooth
compact surface $\Ss(\Pi_N)$ by substituting every ideal triangle in
$\Pi_N \subset \Ss^\infty(\Pi_N)$ by a compact equilateral hyperbolic
triangle with interior angles $2\pi/N$, and glueing them along their
edges in the same way as the ideal triangles of
$\Ss^\infty(\Pi_N)$. The group of orientation preserving isometries of
$\Ss(\Pi_N)$ preserving this triangulation is, again, isomorphic to
${\rm PSL}(2,\Z_N)$. Hence, the automorphism group of the
triangulation $\Pi_8 \subset \Ss(\Pi_8)$ is ${\rm PSL}(2,\Z_8)$ of
order $192$. This implies that $\Ss(\Pi_8)$ is the unique compact
hyperbolic surface of genus 5 with maximal automorphism group (see
\cite{Battsengel}).

\begin{figure}[htbp]
\begin{center}
\includegraphics[width=\textwidth]{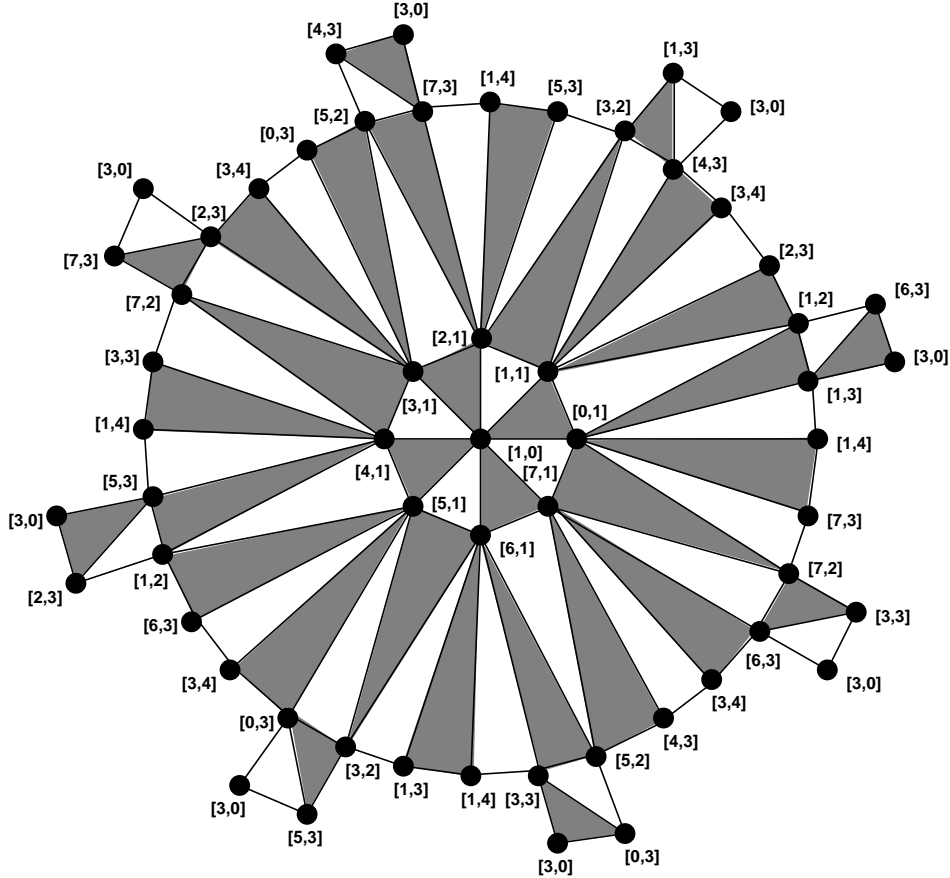}
\caption{The Platonic graph $\Pi_8$: Each triangle corresponds to a
  hyperbolic $(\pi/4,\pi/4,\pi/4)$-triangle of the tessellation of
  $\Ss(\Pi_8)$. The edges along the boundary path are pairwise glued to
  obtain $\Ss(\Pi_8)$.}
\label{fig:s8}
\end{center}
\end{figure}

\subsection{Duality between $T_2$ and $\Pi_8$ in $\Ss_2$}

The $\Pi_8$-triangulation of $\Ss(\Pi_8)$ is illustrated in Figure
\ref{fig:s8}; the black-white pattern on the triangles is a first test
whether this triangulation can be isomorphic to the
$T_2^*$-triangulation of $\Ss_2$. (The $\Pi_N$-triangulations for $3
\le N \le 7$ can be found in Figs. 3 and 4 of \cite{ISi}.)  ${\rm
  PSL}(2,\Z_8)$ acts simply transitively on the directed edges of this
triangulation. Consider now a refinement of this triangulation by
subdividing each $(\pi/4,\pi/4,\pi/4)$-triangle into six
$(\pi/2,\pi/3,\pi/8)$-triangles. It is easily checked that the smaller
$(\pi/2,\pi/3,\pi/8)$-triangles admit also a black-white colouring
such that the neighbours of all smaller black triangles are white
triangles and vice versa (see Figure \ref{fig:triangsubdiv}). Each
black $(\pi/2,\pi/3,\pi/8)$-triangle is in 1-1 correspondence to a
half-edge of $\Pi_8$ which, in turn, can be identified with a directed
edge of $\Pi_8$. Consequently, the orientation preserving isometries
of the surface $\Ss(\Pi_8)$ corresponding to the elements in ${\rm
  PSL}(2,\Z_8)$ act simply transitively on the black
$(\pi/2,\pi/3,\pi/8)$-triangles. In fact, ${\rm PSL}(2,\Z_8)$ can be
interpreted as a quotient of the triangle group $\Delta^+(2,3,8)$,
namely,
$$ {\rm PSL}(2,\Z_8) \cong \langle x^2,y^3,z^8 \mid xyz, 
(xz^2xz^5)^2 \rangle, $$ 
where $x,y,z$ correspond to rotations by $\pi,2\pi/3,\pi/4$ about the
three vertices of a given $(\pi/2,\pi/3,\pi/8)$-triangle.

\begin{figure}[htbp]
\psfrag{A}{$A$}
\psfrag{B}{$B$}
\psfrag{C}{$C$}
\psfrag{Pi8}{$\pi/8$}
\psfrag{Pi6}{$\pi/6$}
\begin{center}
\includegraphics[height=6cm]{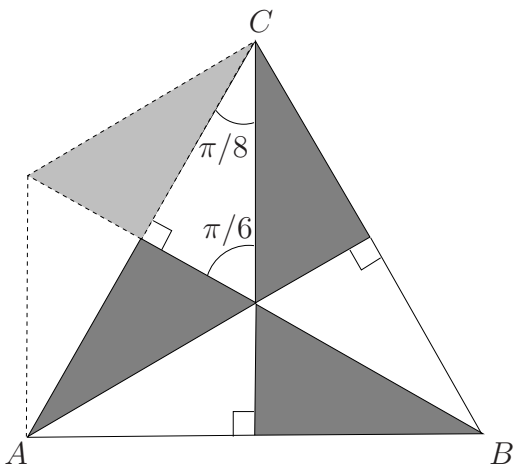}
\caption{Subdivision of the $(\pi/4,\pi/4\,\pi/4)$-triangle $\Delta
  ABC$ into six $(\pi/2,\pi/3,\pi/8)$-triangles with black-white
  colouring.}
\label{fig:triangsubdiv}
\end{center}
\end{figure}

MAGMA computations show that ${\rm PSL}(2,\Z_8)$ has a unique normal
subgroup $N$ of index 6, generated by the elements $X = x^{-1}z^2x$,
$Y = y^{-1}z^2y$ and $Z = z^2$, which is isomorphic to the triangle
group quotient $\Delta^+(4,4,4)/P_2(\Delta^+(4,4,4))$ via the explicit
isomorphism
\begin{equation} \label{eq:XYZ} X \mapsto \begin{pmatrix} -1 & 0 \\ 2
    & -1 \end{pmatrix}, \quad Y \mapsto \begin{pmatrix} -1 & 2 \\ -2 &
    3 \end{pmatrix}, \quad Z \mapsto \begin{pmatrix} 1 & 2 \\ 0 &
    1 \end{pmatrix},
\end{equation}
where $P_2(\Delta^+(4,4,4))$ is a group in the lower exponent-2 series
of $\Delta^+(4,4,4)$ and was defined in \eqref{eq:PkG}. Note that the
matrices in \eqref{eq:XYZ}, viewed as elements in ${\rm PSL}(2,\Z)$,
generate a group acting simply transitively on the black triangles of
the Farey tessellation in $\Hh^2$, as illustrated in Figure
\ref{fig:fargen}. The images of a black triangle ${\mathcal T}$ with
vertices $0,1,\infty$ under $\{ X^{\pm 1}, Y^{\pm 1}, Z^{\pm 1} \}$
are the six black triangles each sharing a common white triangle with
${\mathcal T}$.

\begin{figure}[htbp]
\begin{center}
\includegraphics[width=0.8\textwidth]{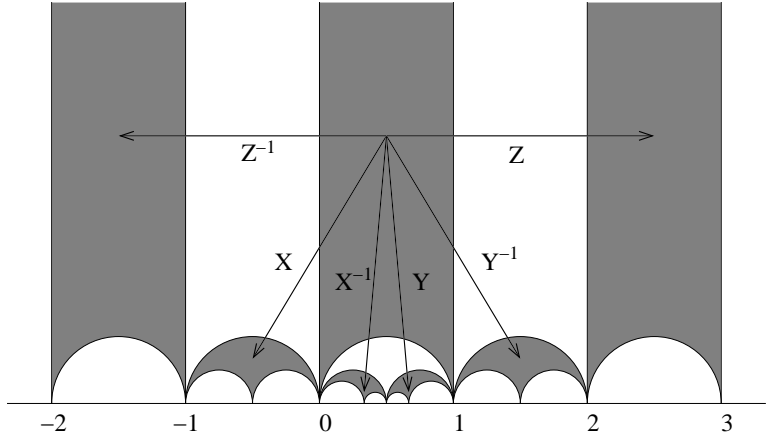}
\caption{The action of the elements $X^{\pm 1},Y^{\pm 1},Z^{\pm 1} \in
  PSL(2,\Z)$ on a triangle ${\mathcal T}$ with vertices $0,1,\infty$
  of the Farey tessellation.}
\label{fig:fargen}
\end{center}
\end{figure}

MAGMA computations also show that we have the explicit isomorphism
$$ N = \langle X,Y,Z \rangle \cong G_2 = \langle x_0,x_1,x_3 \rangle, $$
given by $X \mapsto x_0, Y \mapsto x_1, Z \mapsto x_3$. The normal
group $N \triangleleft {\rm PSL}(2,\Z_8)$ is of order 32 and the
quotient $\Ss_0 = \Ss(\Pi_8)/N$ is an orbifold consisting of two
hyperbolic $(\pi/4,\pi/4,\pi/4)$-triangles (one of them black and the
other white). We conclude from the explicit isomorphism $N \cong G_2$
that the covering procedure discussed in Section \ref{subsec:surfs}
leads to isometric surfaces $\Ss(\Pi_8) \cong \Ss_2$, and that the
tessellation $\Pi_8 \subset \Ss(\Pi_8)$ is dual to the tessellation
$T_2 \subset \Ss(T_2)$ via this isometry of surfaces. This confirms
the first statement of Proposition \ref{prop:isomt2s8thm:main}.




\begin{rmk} For $k=2$, the $(\Delta-Y)$-transformation $X_2 \to T_2$
  has a group theoretical interpretation. There exists a group
  extension $\widetilde{G_2}$ of $G_2$ by $\Z_2$, generated by
  involutions $A,B,C$ satisfying $X = AB$, $Y = BC$ and $Z = CA$, and
  $T_2$ is the Cayley graph of $\widetilde{G_2}$ with respect to the
  generators $A,B,C$. This group theoretic interpretation of the
  $(\Delta-Y)$-transformation {\em fails} for $k \ge 5$.  In fact, the
  group
  $$ T = \langle A,B,C \mid A^2, B^2, C^2,r_1(AB,BC),r_2(AB,BC),r_3(AB,BC) 
  \rangle $$ with $r_1,r_2,r_3$ given in \eqref{eq:r1r2r3} is finite
  and of order 6144. If the introduction of the above involutions
  $A,B,C$ would lead to a group extension $\widetilde{G_k}$, then
  $\widetilde{G_k}$ would have to be of order $2 |G_k|$ and a quotient
  of $T$ and, therefore, of order $\le 6144$. However, we have $2
  |G_5| = 16384$ in contradiction to the second condition. Thus we do
  not obtain a Cayley graph representation of the graphs $T_k$ for $k
  \ge 5$ via this procedure.
\end{rmk}

\subsection{Non-duality of $T_k \subset \Ss_k$ and Platonic
  graphs for $k \ge 3$}

Let $V(\Pi_N)$ denote the vertex set of $\Pi_N$. Then we have
$$ |V(\Pi_N)| = \frac{N^2}{2} \prod_{p | N} \left( 1 - \frac{1}{p^2} 
\right), $$ 
where the product runs over all primes $p$ dividing $N$. This formula
can be found in \cite[p. 441]{ISi}, where $\Pi_N$ is viewed as a
triangular map and denoted by ${\mathcal M}_3(N)$. An isomorphism
$T_k^* \cong \Pi_N$ leads to $N = 2^{n_k+1}$, since all vertices of
$T_k^*$ have degree $2^{n_k+1}$ (see Proposition \ref{prop:combprops})
and all vertices of $\Pi_N$ have degree $N$. In this case, the formula
for the number of vertices of $\Pi_N$ simplifies to
$$ |V(\Pi_{2^{n_k+1}})| = \frac{2^{2n_k+2}}{2} \left( 1 - \frac{1}{4} \right)
= 3 \cdot 2^{2n_k-1}. $$
On the other hand, if $V(T_k^*)$ denotes the vertex set of $T_k^*$, we
conclude from Proposition \ref{prop:combprops},
$$ |V(T_k^*)| = 3 \cdot 2^{N_k-n_k}. $$
Hence, an isomorphism $T_k^* \cong \Pi_N$ leads to the identity
$2n_k-1 = N_k-n_k$, i.e.,
$$ 3 \lfloor \log_2 k \rfloor +3 =  3n_k = N_k+1 
\ge 8 \lfloor k/3 \rfloor + 3 \cdot (k\, {\rm mod}\, 3), $$ by
\eqref{eq:n_k} and \eqref{eq:N_k}. But one easily checks that this
inequality holds only for $k=1,2$. (In the case $k=1$, we have $\Pi_4
= T_1^*$, since $T_1$ is combinatorially the cube and $\Pi_4$ is the
octagon.) This shows that the graph family $\Pi_N$ cannot contain any
of the dual graphs $T_k^*$, for indices $k \ge 3$. This finishes the
proof of Proposition \ref{prop:isomt2s8thm:main}.

\section{Spectral properties of the closed hyperbolic surfaces 
$\Ss_k$ and $\widehat \Ss_k$}


\subsection{A lower eigenvalue estimate for the surfaces $\Ss_k$}
\label{subsec:corlambda}

The goal of this section is to prove Corollary
\ref{cor:lowerlambda1}. Let us start with lower estimates on the first
nontrivial Neumann eigenvalues of special bounded sets
$S \subset {\mathbb H}^2$ with piecewise smooth boundaries,
characterized by
\begin{equation} \label{eq:nuS} 
\nu(S) := \inf\left\{ \int_S \Vert {\rm grad}\, f \Vert^2 \Big\vert 
\text{$f \in C^\infty(S)$ with $\int_S f^2 = 1$ and $\int_S f = 0$} 
\right\}. 
\end{equation}

\begin{lem} \label{lem:quad} For $k \ge 2$, let
  ${\mathcal Q}(k) \subset {\mathbb H}^2$ be a quadrilateral obtained
  as the union of two adjacent equilateral hyperbolic triangles with
  interior angles $\pi/2^k$. Then we have for all $k \ge 2$,
  $$ \nu({\mathcal Q}(k)) > \frac{1}{4}. $$
\end{lem}

\begin{proof} The result follows via Cheeger's inequality (see, e.g.,
  \cite[Thm. 8.3.3]{Bu}), once it is shown that
  $$ \frac{1}{4} \le h({\mathcal Q}(k)) = \inf \frac{\ell(\gamma)}
  {\min\{{\rm area}\, A, {\rm area}\, A'\}}, $$
  where $\gamma$ runs through all curves decomposing ${\mathcal Q}(k)$ into
  two connected, relatively open subsets $A$ and $A'$. Then we must
  have one of the following cases: (i) $\gamma$ is a closed curve, (ii) $\gamma$
  is an arc with both endpoints on the same side of ${\mathcal Q}(k)$,
  (iii) $\gamma$ is an arc having endpoints in adjacent sides of
  ${\mathcal Q}(k)$, and (iv) $\gamma$ is an arc having endpoints in
  opposite sides of ${\mathcal Q}(k)$. 

  \begin{figure}[htbp]
    \psfrag{A}{$A$}
    \psfrag{B}{$B$}
    \psfrag{C}{$C$}
    \psfrag{D}{$D$}
    \psfrag{E}{$E$}
    \psfrag{F}{$F$}
    \psfrag{c}{$c$}
    \psfrag{c'}{$c'$}
    \psfrag{a}{$\alpha$}
    \psfrag{a/2}{$\alpha/2$}
    \psfrag{h}{$h$}
    \begin{center}
      \includegraphics[height=5cm]{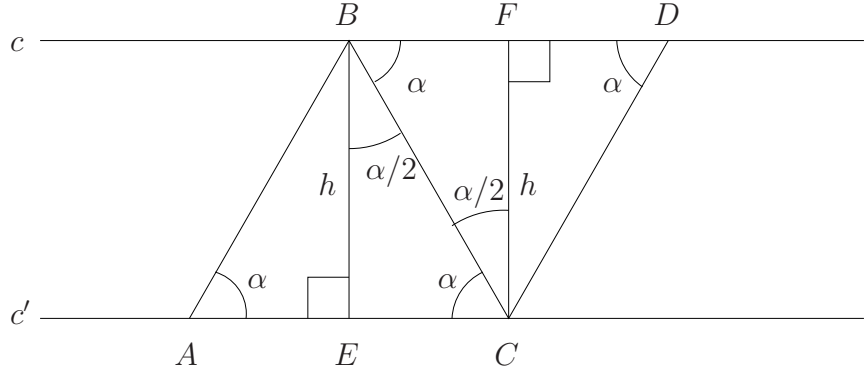}
      \caption{The quadrilateral ${\mathcal Q}(k)$ as union of the two
        triangles $\Delta ABC$ and $\Delta BCD$.}
      \label{fig:quad1}
    \end{center}
  \end{figure}

  Due to \cite[p.220]{Bu}, the only remaining case to consider is case
  (iv). Let $\alpha = \pi/2^k$ and
  ${\mathcal Q}(k) = \triangle ABC \cup \triangle BCD$, as illustrated in
  Figure \ref{fig:quad1}. We assume the endpoints of $\gamma$ lie on the
  opposite sides $AC$ and $BD$, whose bi-infinite geodesic extensions
  are $c$ and $c'$. Then we have obviously $\ell(\gamma) \ge d(c,c')$. If
  $h$ denotes the length of the heights of $\triangle ABC$ and
  $\triangle BCD$, then we conclude from \cite[Thm. 2.2.2]{Bu} that
  \begin{equation} \label{eq:coshh} 
  \cosh(h) = \frac{\cos(\alpha)}{\sin(\alpha/2)}. 
  \end{equation}
  It is easy to see from $3\alpha/2 < \pi/2$ that if $B_t$ denotes the
  point on $c$ at distance $t = d(B,B_t)$ from $B$ in the direction of
  $F$, then $t \to d(B,B_t)$ is initially strictly monotone
  decreasing. The same holds for the point $C_t \in c'$ at distance
  $t$ from $C$ in the direction of $E$. As a consequence, the unique
  minimal distance points $X \in c$ and $X' \in c'$, i.e.,
  $d(X,X') = d(c,c')$, must lie in the segments $BF$ and $EC$,
  respectively, and the quadrilateral $\square XX'CF$ is a
  trirectangle. Then we conclude from \cite[Thm. 2.3.1]{Bu} that
  $$ \cosh(d(X,X')) = \cosh(h) \sin(3\alpha/2). $$
  Combining this with \eqref{eq:coshh}, we obtain, using
  $\alpha \le \frac{\pi}{4}$:
  $$ \cosh(\ell(\gamma)) \ge \frac{\cos(\alpha) \sin(3\alpha/2)}{\sin(\alpha/2)}
  = \cos(\alpha) (1+2\cos(\alpha) ) \ge \frac{1+\sqrt{2}}{\sqrt{2}},$$
  i.e., $\ell(\gamma) \ge 1.12838\dots$. Since
  ${\rm area}\, \triangle ABC = \pi-3\alpha \le \pi$, we conclude
  that
  $$ \frac{\ell(\gamma)}{\min\{ {\rm area}\, A, {\rm area}\, A' \}} \ge 
  \frac{1.12838}{\pi} > \frac{1}{4}, $$
  finishing the proof of $\nu({\mathcal Q}(k)) > 1/4$. 
\end{proof}

Let ${\mathcal S} = {\mathcal S}_k$. Note that ${\mathcal S}$ comes
with a tessellation $T_k = (V_k,E_k,F_k)$ by regular polygons. To
bound the first nontrivial eigenvalue $\lambda_1({\mathcal S})$ of
$\mathcal S$ from below by the spectral gap $\sigma(T_k)$ of the
tessellation $T_k = (V_k,E_k.F_k)$, we employ the Brooks-Burger
transfer principle and follow closely the arguments given in
\cite{Breu}, which are a variant of Burger's arguments \cite{Bur}. For
the reader's convenience, we present them here.

\begin{proof}[Proof of Corollary \ref{cor:lowerlambda1}]
Let ${\mathcal S} = {\mathcal S}_k$ and $f \in C^\infty({\mathcal S})$ be 
the normalized eigenfunction to the eigenvalue $\lambda_1 = 
\lambda_1({\mathcal S})$, i.e.,
$$ \Delta f = \lambda_1 f \qquad \text{and} \quad \int_{\mathcal S} f^2 =1. $$
Then $f$ is orthogonal to the constant function, i.e., $\int_{\mathcal S} f = 0$
and satisfies $\int_{\mathcal S} \Vert {\rm grad}\, f \Vert^2 = \lambda_1$.  

Now we associate to $f$ a corresponding function $F$ on the
set of vertices $V_k$ of the tessellation $T_k \subset {\mathcal
  S}$. For every vertex $v \in V_k$, there is an equilateral hyperbolic
triangle ${\mathcal T}(v) \subset {\mathcal S}$ with interior angles
$\pi/2^{n_k}$ of the dual tessellation containing $v$. In fact, we
have
$$ {\mathcal T}(v) = \{ z \in {\mathcal S} \mid 
d(z,v) \le d(z,w)\ \text{for all $w \in V_k$} \}. $$
The function $F: V_k \to {\mathbb R}$ is now defined as follows
$$ F(v) = \frac{1}{V} \int_{{\mathcal T}(v)} f, $$
where $V = {\rm area}({\mathcal T}(v))$, 
and we have $V \sum_{v \in V_k} F(v) = \int_{\mathcal S} f = 0$.

Our next goal is to compare the Rayleigh quotients of $f$ and $F$. The
characterisaion \eqref{eq:nuS} implies that we have the following Poincar{\'e} inequalities
for adjacent vertices $v,w \in V_k$:
\begin{eqnarray*}
\int\limits_{{\mathcal T}(v)} (f-F(v))^2 &<& \frac{1}{\nu({\mathcal T}(v))} 
\int\limits_{{\mathcal T}(v)} \Vert {\rm grad}\, f \Vert^2, \\
\int\limits_{{\mathcal T}(v) \cup {\mathcal T}(w)} \left( f - \frac{F(v)+F(w)}{2} 
\right)^2 &<& \frac{1}{\nu({\mathcal T}(v) \cup {\mathcal T}(w))}  
\int\limits_{{\mathcal T}(v) \cup {\mathcal T}(w)} \Vert {\rm grad}\, f \Vert^2.
\end{eqnarray*}
Using \cite[(8.4.1)]{Bu} and Lemma \ref{lem:quad}, we conclude that
both inequalities above hold by chosing the same coefficient $1/ \nu$
with $\nu = 1/4$ at the right hand sides.

Using
$$ \frac{1}{2} \left( \frac{F(v)-F(w)}{2} \right)^2 \le 
\left( f(z) - \frac{F(v)+F(w)}{2} \right)^2 + ( f(z) - F(v) )^2 $$
and the Poincar{\'e} inequalities, we obtain
\begin{multline*}
V \left( \frac{F(v)-F(w)}{2} \right)^2 \\ = 
\int_{{\mathcal T}(v)} \frac{1}{2} \left( \frac{F(v)-F(w)}{2} \right)^2 +
\int_{{\mathcal T}(w)} \frac{1}{2} \left( \frac{F(v)-F(w)}{2} \right)^2 \\ < 
\frac{2}{\nu} \int_{{\mathcal T}(v) \cup {\mathcal T}(w)} \Vert {\rm grad}\, f \Vert^2,
\end{multline*}
leading to
\begin{equation} \label{eq:gradF} 
V \sum_{\{v,w\} \in E_k} (F(v)-F(w))^2 < 
\frac{8 D}{\nu} \int_{\mathcal S} \Vert {\rm grad}\, f \Vert^2 = 
8 D \frac{\lambda_1}{\nu} 
\end{equation}
after summation over the edges, and using the fact that $T_k$ has
vertex degree $D=3$. On the other hand, using the first Poincar{\'e}
inequality again we have
$$ V (F(v))^2 = \int_{{\mathcal F}(v)} f^2 - \int_{{\mathcal F}(v)} (f-F(v))^2
> \int_{{\mathcal F}(v)} f^2 - \frac{1}{\nu} 
\int_{{\mathcal F}(v)} \Vert {\rm grad}\, f \Vert^2, $$
and, after summing over the vertices of $T_k$, 
\begin{equation} \label{eq:F} 
V \sum_{v \in V_k} (F(v))^2 > 1 - \lambda_1/\nu. 
\end{equation}
Combining \eqref{eq:gradF} and \eqref{eq:F}, we have either
$\lambda_1 \ge 1/4 = \nu$ or
$$ \sigma(T_k) \le \frac{\sum_{\{v,w\} \in E_k} (F(v)-F(w))^2}
{\sum_{v \in V_k} (F(v))^2} < \frac{8 D \lambda_1/\nu}{1-\lambda_1/\nu}, $$
which implies, in either case,
$$ \lambda_1 > \nu\left( \frac{\sigma(T_k)}{8D+\sigma(T_k)} \right) =
\frac{1}{4}\left( \frac{\sigma(T_k)}{24+\sigma(T_k)} \right) . $$

The Corollary follows now directly from Theorem \ref{thm:main}(ii)
since $\sigma(T_k) = 3- \sqrt{C+3} > 0$ with the constant $C \in (0,6)$
given in the theorem.
\end{proof}

\subsection{Conformal compactifications}
\label{subsec:confcomp}

We start with the definition of a conformal compactification of a
hyperbolic surface as given, e.g., in \cite{Br3}.

\begin{dfn} Let ${\mathcal S}^\infty$ be a complete non-compact finite
  area hyperbolic surface with $k$ cusps. The {\em conformal
    compactification of ${\mathcal S}^\infty$} is the unique closed
  Riemann surface $\mathcal S$ with $k$ points
  $\{ p_1, \dots, p_k \} \subset {\mathcal S}$ such that
  ${\mathcal S}^\infty$ is conformally equivalent to
  ${\mathcal S} \backslash \{p_1,\dots,p_k\}$.
\end{dfn}

The conformal compactification of ${\mathcal S}^\infty$ can be
constructed as follows: Choose disjoint neighbourhoods of the cusps of
${\mathcal S}^\infty$ such that every such neighbourhood is
conformally equivalent to a punctured disk.  Fill in each disk the
missing point and choose the unique conformal structure of this disk,
and you obtain a compact surface $\mathcal S$ conformally equivalent
to ${\mathcal S} \backslash \{p_1,\dots,p_k\}$. For more details and
many explicit examples, we refer the readers to \cite{Br3,BrM,Man}.

It is not always the case that the conformal compactification
$\mathcal S$ carries a hyperbolic metric. For example, the conformal
compactification of the hyperbolic ${\mathcal S}_0^\infty$ in Section
\ref{subsec:surfs}, obtained by glueing together two ideal hyperbolic
triangles ${\mathcal T}_1^\infty$ and ${\mathcal T}_2^\infty$, is the
Riemann sphere which only carries a metric of positive constant
curvature. A sufficient condition to guarantee that the conformal
compactification $\mathcal S$ carries a hyperbolic metric which, in an
appropriate sense, is even geometrically close to the hyperbolic
structure of ${\mathcal S}^\infty$, is the so-called {\em large cusp
  condition} (see, e.g., \cite[Def. 2.1]{Br3} or
\cite[Def. 2.2.1]{Man}):

\begin{dfn} A hyperbolic surface ${\mathcal S}^\infty$ {\em has cusps of
  length $\ge L$} if, for every cusp of ${\mathcal S}^\infty$, there
  exists a closed horocycle of length $\ge L$ about it, and if all
  these horcycles are disjoint.
\end{dfn}

In fact, it can be shown (see \cite[Thm. 2.3.1]{Man}) that if all
cusps of ${\mathcal S}^\infty$ have length bigger than $2\pi$, then
its conformal compactification carries a hyperbolic metric. Following
arguments explained in detail in \cite{Man}, we show that our closed
surfaces ${\mathcal S}_k$ are the conformal compactifications of the
surfaces ${\mathcal S}_k^\infty$.

\begin{prop} \label{prop:confcomp} The closed surfaces
  ${\mathcal S}_k$ of Section \ref{sec:combpropssurf} are the
  conformal compactifications of the complete non-compact finite area
  hyperbolic surfaces ${\mathcal S}_k^\infty$.
\end{prop}

\begin{proof} Let $k \ge 2$ be fixed. Recall that the group
  $G_k$ acts on ${\mathcal S}_k^\infty$ by isometries and that
  ${\mathcal S}_k^\infty$ has a geodesic triangulation by (black and
  white) ideal triangles. We noted at the end of Section
  \ref{subsec:dirconst} that the cusps of ${\mathcal S}_k^\infty$ are
  in bijection with left-hand-turn pathes in $(T_k,{\mathcal O}_k)$ of
  length $2^{n_k+1} \ge 8 > 2\pi$ (with $n_k$ given in \eqref{eq:n_k}). This
  implies that the large cusp condition (with cusp lengths increasing
  in $k$) is satisfied and the conformation compactification
  ${\mathcal S}_{\rm comp}$ of ${\mathcal S}_k^\infty$ carries a
  hyperbolic metric with a corresponding geodesic
  triangulation. Moreover, any isometry of ${\mathcal S}_k^\infty$
  induces a corresponding isometry of ${\mathcal S}_{\rm comp}$. For
  more details on these facts see, e.g., \cite[pp. 13]{Man}. Therefore
  $G_k$ acts also on ${\mathcal S}_{\rm comp}$ by isometries, and
  ${\mathcal S}_{\rm comp} / G_k$ is an orbifold with a triangulation
  consisting of two compact triangles whose vertices are the singular
  points. The orders of the generators $x_0, x_1, x_3$ of $G_k$
  determine the angles of these two triangles uniquely, and we
  conclude that ${\mathcal S}_{\rm comp} / G_k$ is isometric to the
  orbifold ${\mathcal S}_0 = {\mathcal S}_k / G_k$ introduced in
  Section \ref{subsec:surfs}. This isometry lifts then to a
  corresponding isometry between ${\mathcal S}_k$ and
  ${\mathcal S}_{\rm comp}$.
\end{proof}

We like to mention the following consequence of Proposition
\ref{prop:confcomp}.

\begin{cor} \label{cor:cheegSk} There are uniform lower positive bounds
  for the Cheeger constants of the families of surfaces
  ${\mathcal S}_k$ and ${\mathcal S}_k^\infty$.
\end{cor}

\begin{proof} From a classical result by Tanner \cite{Tan} or
  Alon-Milman \cite{AM} we know that
  $$ \frac{\sigma(T_k)}{2} \le h(T_k) := \inf_E 
  \frac{\#(E)}{\min\{\#(A),\#(A')\}}, $$
  where $E \subset E_k$ runs through all collection of edges such that
  $T_k \backslash E_k$ disconnects into two components with disjoint
  vertex sets $A \subset V_k$ and $A' \subset V_k$. This implies
  together with Theorem \ref{thm:main}(ii) that the combinatorial
  Cheeger constants have the following uniform positive lower bound
  \begin{equation} \label{eq:h_0} 
  h_0 = \frac{3-\sqrt{C+3}}{2} \le h(T_k). 
  \end{equation}
  Moreover, it follows from Theorem 4.2 in \cite{BrM} that there are
  constants $c_h, C_h > 0$ such that we have the following relation
  between the respective Cheeger constants
  $$ c_h h(T_k) \le h({\mathcal S}_k^\infty) \le C_h h(T_k). $$
  From these facts we conclude that
  \begin{equation} \label{eq:Skinfest} 
  c_h h_0 \le h({\mathcal S}_k^\infty). 
  \end{equation}
  Moreover, it follows from Theorem 3.3(a) in \cite{BrM}, Proposition
  \ref{prop:confcomp} above, and the increasing cusp length properties
  of our surfaces ${\mathcal S}_k^\infty$ that, for every
  $\epsilon > 0$, there exists $k_0 \in {\mathbb N}$ such that
  \begin{equation} \label{eq:hSkSkinf} 
  \frac{1}{1+\epsilon} h({\mathcal S}_k) \le h({\mathcal S}_k^\infty)
  \le (1+\epsilon) h({\mathcal S}_k) 
  \end{equation}
  for all $k \ge k_0$. Combining, finally, \eqref{eq:Skinfest} and
  \eqref{eq:hSkSkinf} proves the corollary.
\end{proof}

\begin{rmk}
  We could have proved Corollary \ref{cor:lowerlambda1} alternatively
  by combining Corollary \ref{cor:cheegSk} and Cheeger's inequality
  for surfaces, but the proof given in Section \ref{subsec:corlambda}
  is much more direct.
\end{rmk}

\subsection{A lower eigenvalue estimate for the surfaces 
${\widehat {\mathcal S}}_k$}
\label{subsec:corlambda2}

It only remains to prove Corollary \ref{cor:lowerlambda2} from the
Introduction. The explicit construction of the surfaces
${\widehat{\mathcal S}}_k$ is explained in Buser \cite[Section
3.2]{Bu1}.

\begin{proof}[Proof of Corollary \ref{cor:lowerlambda2}]
The identity $2g-2 = |V_k|$ between the genus of the surface $\widehat
\Ss_k$ and the number of vertices of the trivalent graph $T_k$ is
easily checked. Moreover, every automorphism of the graph $T_k$
induces an isometry on $\widehat \Ss_k$. Since the graphs $T_k$ form a
power of coverings with powers of $2$ as covering indices, the same
holds true for the associated surfaces $\widehat \Ss_k$. We know from
\cite[(4.1)]{Bu1} that
$$ \lambda_1(\widehat \Ss_k) \ge \frac{1}{144\pi^2} h(T_k). $$
Combining this with \eqref{eq:h_0} leads to
$$ \lambda_1(\widehat \Ss_k) \ge \frac{3-\sqrt{C-3}}{288 \pi^2} $$
with the constant $C>0$ from Theorem \ref{thm:main}. This finishes the
proof of Corollary \ref{cor:lowerlambda2}.
\end{proof}

\end{document}